\documentclass[a4paper,preprint,11pt]{amsart}

\usepackage{amsmath, amssymb}
\usepackage{amsthm}
\usepackage[colorlinks]{hyperref}
\usepackage{setspace}
\usepackage{color}
\usepackage{lmodern} 
\usepackage[OT1]{fontenc}
\usepackage[top=1in,bottom=1in,left=1in,right=1in]{geometry}
\allowdisplaybreaks

\newcommand{\rE}{\mathbb{E}}

\newcommand{\rP}{\mathbb{P}}

\newcommand{\re}{\mathrm{e}}
\newcommand{\ep}{\epsilon}

\newcommand{\ri}{\mathrm{i}}
\newcommand{\rd}{\mathrm{d}}

\newcommand{\bb}{\begin{eqnarray*}}
\newcommand{\ee}{\end{eqnarray*}}
\newcommand{\bbb}{\begin{eqnarray}}
\newcommand{\eee}{\end{eqnarray}}

\DeclareMathOperator{\var}{\mathrm{var}}

\theoremstyle{plain}
\newtheorem{theorem}{Theorem}
\newtheorem{lemma}[theorem]{Lemma}
\newtheorem{proposition}[theorem]{Proposition}
\newtheorem{corollary}[theorem]{Corollary}
\theoremstyle{definition}
\newtheorem{example}[theorem]{Example}

\begin{document}
\title[Self-intersection local times]{Optimal bounds for self-intersection local times}
\author{George Deligiannidis}
\address{Department of Statistics, University of Oxford, Oxford OX1 3TG, UK}
\email{deligian@stats.ox.ac.uk}
\author{Sergey Utev}
\address{Department of Mathematics, University of Leicester, LE1 7RH,UK}
\email{su35@le.ac.uk}
\begin{abstract}
For a random walk $S_n, n\geq 0$ in $\mathbb{Z}^d$, let $l(n,x)$ be its local time at the site
$x\in \mathbb{Z}^d$. Define the $\alpha$-fold self intersection local time $L_n(\alpha) := \sum_{x} l(n,x)^{\alpha}$, and let $L_n(\alpha|\ep, d)$ the corresponding  quantity for  $d$-dimensional simple random walk.
Without imposing any moment conditions, we show that the 
variances of the local times $\var(L_n(\alpha))$ of any genuinely $d$-dimensional random walk are bounded above by 
the corresponding characteristics of the simple symmetric random walk in $\mathbb{Z}^d$, i.e. $\var(L_n(\alpha)) \leq C \var[L_n(\alpha|\ep, d)]\sim K_{d,\alpha}v_{d,\alpha}(n)$. 
In particular, variances of local times of all genuinely $d$-dimensional random walks, $d\geq 4$, are similar 
to the $4$-dimensional symmetric case $\var(L_n(\alpha)) = O(n)$.  On the other hand, in dimensions $d\leq 3$ the resemblance to the simple 
random walk $\liminf_{n\to \infty} \var(L_n(\alpha))/v_{d,\alpha}(n)>0$ implies that the jumps must have zero mean and finite second moment.
\end{abstract}

\maketitle

\section{Introduction and main results}
Let $X, X_1, X_2, \dots$ be independent,  identically distributed, $\mathbb{Z}^d$-valued random variables, 
and define the random walk $S_0:=0$, $S_n = \sum_{j=1}^n X_j$, for $n\geq 1$
Let $l(n,x)=\sum_{j=1}^n\mathbf{I}(S_j=x)$ be the local time of $(S_n)_n$ at the site $x \in \mathbb{Z}^d$,
and define for a positive integer $\alpha$ 
the $\alpha$-fold {\sl self-intersection local time}
$$L_n=L_n(\alpha) = \sum_{x\in\mathbb{Z}^d} l(n,x)^{\alpha} =
\sum_{i_1, \dots, i_\alpha = 0}^n \mathbf{I}(S_{i_1} = \cdots = S_{i_\alpha} ).$$
Our method also applies to the more general case where the $X_i$ are independent but not identically distributed. 
To distinguish between the two cases, we shall refer to random walk with independent identically distributed increments as the i.i.d.\ case.
Following Spitzer~\cite{Spitzer76}, in the i.i.d.\ case, we call $X_i$ and the random walk it generates 
{\sl genuinely $d$-dimensional} if the support of the variable $X_1-X_2$
linearly generates $d$-dimensional space.  Finally let  $\Gamma=[0,2\pi]^d$.

The quantity $L_n(\alpha)$ has received considerable attention in the literature due to its
relation to {\sl self-avoiding walks}  and {\sl random walks in random scenery}. In particular let the {\sl random scenery}
$\{\xi_{x}, x\in \mathbb{Z}^d\}$ be a collection of i.i.d.\ random variables, independent of the $X_i$, 
and   define the process $Z_0 =0$, $Z_n = \sum_{i=1}^n \xi_{S_i}.$ Then $(Z_n)_n$ is commonly referred
to as {\sl random walk in random scenery} and was introduced in Kesten and Spitzer~\cite{Kest79}, 
where functional limit  theorems were obtained for $Z_{[nt]}$ under an appropriate normalization 
for the case $d=1$. The case  $d=2$, with $X_i$ centered with non-singular covariance matrix, was treated in \cite{Bolt89} where
it was shown that $Z_{[nt]}/\sqrt{n\log n}$ converges weakly to Brownian motion.
As  is obvious from the identities $Z_n = \sum_{x\in \mathbb{Z}^d} l(n,x) \xi_x$, $\var(Z_n) = \var[L_n(2)]  \var(\xi_x)$, 
limit theorems for $Z_n$ usually require asymptotics for the local times of the random walk $(S_n)_n$.


Such asymptotics are usually obtained from Fourier techniques 
applied to the characteristic function $f(t)=\rE[\exp(\ri t\cdot X)]$ under the additional assumption of a 
Taylor expansion of the form $f(t)=1- \langle \Sigma t, t\rangle+o(|t|^2)$ 
where $\Sigma$ is the positive definite covariance matrix 
\cite{Bolt89, Bor79, Bry95, GS09,Cerny07}, which further requires that $\rE|X|^2<\infty$ and $\rE X=0$. 
Similar restrictions are also required for the application of local limit theorems such as in \cite{Law91, Lewis93}.

In this paper, motivated by the results of Spitzer~\cite{Spitzer76} for genuinely $d$-dimensional random walks
and the approach of Becker and K\"onig~\cite{Beck09}(see also Asselah~\cite{Asselah} where non-integer $\alpha$ is also treated) we shall study the asymptotic behavior of $\var(L_n(\alpha))$
without imposing any moment assumptions on the random walk. The central idea behind our approach is 
to compare the self-intersection local times $L_n(\alpha)$ of a general $d$-dimensional walk with those of its symmetrised version. 
In addition we will compare the self-intersection local times of a general $d$-dimensional random walk with those of the $d$-dimensional simple symmetric random walk, 
$S_n^{\ep,d}$  which we denote by $L_n(\alpha|\ep,d)$.  
Recall that simple random walk in $\mathbb{Z}^d$ is defined as
$S_0^{\ep,d}:=0$, $S_n^{\ep,d}:= \sum_{j=1}^n X_{j}^{\ep,d}$ for $n\geq 1$,
where for $k=1,\dots, d$ $\rP(X_j^{\ep,d}= \pm e_k) = 1/2d$ and $e_k$ is the $k$-th unit coordinate vector.
It is well-known that with some positive constant $K_{\alpha,d}$,  
$\var[L_n(\alpha|\ep,d)] \sim K_{\alpha,d} v_{d,\alpha}(n)$ where 
\begin{eqnarray*}
v_{1,\alpha}(n) = n^{1+\alpha}, \quad v_{2,\alpha}(n) = n^2 \log(n)^{2\alpha -4}, \quad \;v_{3,\alpha}(n) = n\log(n)\quad \mbox{ and }\; v_{d,\alpha}(n) = n, \quad d\geq 4.\end{eqnarray*}
Several other cases have been treated in the literature, using a variety of methods. 

A careful look at the literature reveals that the most difficult case in $d=2$ 
is the {\it{near transient recurrent}} case, where $\rP(S_n=0)\sim C/n$, which corresponds to genuinely 
$2$-dimensional symmetric recurrent random walks, which will be referred to as a critical case.
%
Surprisingly enough, the variance of the self-intersection local times in the critical case is asymptotically the largest.
\begin{theorem}\label{thm:dbound}
Let $X_i$ be i.i.d., genuinely $d$-dimensional. Then,
$$\var(L_n(\alpha)) \leq c_{\alpha,X}  \var\big(L_n(\alpha|\ep,d)\big)\leq C_{\alpha,X} v_{d,\alpha}(n)\;.$$
\end{theorem}
The result was motivated by \cite{Spitzer76} and \cite{Beck09} (and improves related results of Becker and Konig for $d=3$ and $d=4$). 
Several cases treated in \cite{Asselah,Bolt89,Bor79,Chen10,DU11,CGP,Beck09,Lewis93} can then be obtained as particular cases.

Moreover, we also show the surprising reverse, more exactly 
that the right asymptotic of $\var(L_n)$ implies that the jumps must have zero mean and finite second moment.
\begin{theorem}\label{thm:nasc}
Let $X_i$ be i.i.d., genuinely $d$-dimensional and $d=1,2,3$. If
$$
\liminf_{n\to \infty} \frac{\var(L_n(\alpha))}{v_{d,\alpha}(n)} > 0,$$ 
then $\rE|X|^2<\infty$ and $\rE X=0$.
\end{theorem}
As it follows from Theorem \ref{thm:asymptotic}, given below, for $d=2,3$ and Theorem 5.2.3 in Chen \cite{Chen10} for $d=1$, 
if $\rE X=0$ and $0<\rE|X|^2<\infty$, then $\liminf_n \var[L_n(\alpha)]/v_{d,\alpha}(n)>0$.
%

For general genuinely  $d$-dimensional random walks with finite second moments and zero mean, the asymptotic 
behavior is similar to $d$-dimensional simple symmetric random walk, again the most complicated case being $d=2$.
Also, as it follows from our general bounds (see Proposition  \ref{pr:main} and Corollary \ref{cor_stan2}), 
the asymptotics for the 
genuinely $d$-dimensional random walk can be reproduced by those of 
the symmetric one-dimensional random walk with appropriately chosen heavy tails, as was indicated by 
Kesten and Spitzer~\cite{Kest79}. The proofs are based on adapting the Tauberian approach developed in \cite{DU11}.
\begin{theorem}\label{thm:asymptotic}
Let $d=1,2,3 $, and suppose that for $t\in[-\pi,\pi]^d$ we have
\begin{align}
f(t) = 1- \gamma |t| + R(t), \,\, \mbox{for $d=1$, or} \qquad
f(t) = 1- \langle \Sigma t, t\rangle + R(t), \,\, \mbox{for $d=2,3$,}
\label{eq:expansion1}
\end{align}
where $\Sigma$ is a non-singular covariance matrix and 
$R(t) = o(|t|)$ for $d=1$, and $o(|t|^2)$ for $d=2,3$ as $t\to 0$.
Then 
\begin{equation*}
\var(L_n (\alpha)) 
\sim 
\begin{cases}
\frac{(\pi^2 + 6)}{12}
\frac{(\alpha !)^2 (\alpha -1 )^2}{(\gamma \pi)^{2\alpha-2}} n^2 \log(n)^{2\alpha -4}
,&\mbox{for $d=1$,}\\
\frac{(\alpha !)^2 (\alpha -1 )^2}{2\big(2 \pi \sqrt{|\Sigma|} \big)^{2\alpha-2}} n^2 \log(n)^{2\alpha -4} ( \kappa +1 )
& \qquad\mbox{for $d=2$, and }\\
 (\kappa_1 + \kappa_2) n \log n,& \mbox{for $d=3$, $\alpha =2$,}
\end{cases}
\end{equation*}
where 
$\kappa = \int_0^\infty\int_0^\infty \rd r \rd s
\Big[(1+r)(1+s) \sqrt{(1+r+s)^2 - 4 r s}\Big]^{-1}- \pi^2/6$ and $\kappa_1$, $\kappa_2$ are defined in \eqref{eq:kappa1} and \eqref{eq:kappa2} respectively.

Moreover, if $L'(n,\alpha)$ is the self-intersection local time of another random walk whose characteristic function also satisfies \eqref{eq:expansion1} then $L'(n,\alpha) =L(n,\alpha) (1+ o(1))$.
\end{theorem}

The methods developed in this paper are used by the first author and K.~Zemer in \cite{DZ15} to prove that the range of 1-stable random walk in $\mathbb{Z}$ and simple random walk in $\mathbb{Z}^2$ has the F\"olner property and therefore to compute the relative complexity of random walk in random scenery in the sense of Aaronson~\cite{Aaronson}.


\section{Proofs}
%
\subsection{General bounds}
We first develop a technique to treat random walks with independent but not necessarily identically distributed increments.
\begin{proposition}\label{pr:main} {\rm (General upper bound)}
Assume that $X_i$ are independent $\mathbb{Z}^d$-valued random variables and let $S_{u,v}:=X_u+\ldots+X_{u+v}$. 
Suppose further that for all $n\in \mathbb{N}$, and integers $a,u, b,v \geq 0$, with $a+u\leq b$ and any $x\in \mathbb{Z}^d$ we have
\begin{align}\label{eq:A}\tag{A}
\rP \Big( S_{a,u} \pm S_{b,v} = x \Big) &\leq \phi(u+v),\\
\label{eq:B}\tag{B}
\rP\Big( S_{a,u} = 0\Big) - \rP\Big( S_{a,u}+ S_{b,v} = 0\Big) &\leq \psi(u,v),
\end{align}
where  $\phi(u)$  is non-{\sl increasing}, 
$\psi(u,v)$ is non-{\sl increasing} in $u$ and is non-{\sl decreasing} and sub-additive in $v$ in the sense that 
$\psi(u, v+w) \leq A_{\psi}\big[ \psi(u,v) + \psi(u, w) \big]$, for some constant $A_{\psi}$ independent of $u, v, w$. 
Then, for some constant $K = cA_{\psi}(1+A_{\psi})^{\alpha -2}$ depending only on $\alpha$
\begin{equation*}
  \var(L_n(\alpha))
  \leq K n \Big( \sum_{i=0}^{n-1} \phi(i) \Big)^{2\alpha -4} \sum_{i,j,k=0}^{n-1}
  \big[ \phi(j\vee i) \phi(k\vee i) + \phi(j) \psi(i+k, j)\big].
\end{equation*}
\end{proposition}
\begin{proof}[Proof  of Proposition~\ref{pr:main}]
We first write out the variance as a sum
\begin{align}
  \var L_n(\alpha)
&= (\alpha !)^2\sum_{k_1 \leq \cdots \leq k_{\alpha}}\sum_{l_1 \leq \cdots \leq l_{\alpha}}
\Big(\rP \big[ S_{k_1} = \cdots = S_{k_{\alpha}}, S_{l_1} = \cdots = S_{l_{\alpha}}\big]\label{eq:var}\\
&\quad -\rP\big[ S_{k_1} = \cdots = S_{k_{\alpha}}\big]\rP \big[S_{l_1} = \cdots S_{l_{\alpha}}\big]\Big)\nonumber.
\end{align}
An important role is played by the manner in which the two sequences are interlaced,
since for example if $k_\alpha \leq l_1$ or $l_\alpha \leq k_1$, the term vanishes by the Markov property.
Let's assume, without loss of generality, that $k_1\leq l_1$ and we arrange the two sequences in an ordered sequence of combined length $2\alpha$ which we denote as
$(p_1, \dots, p_{2\alpha})$; we also define $(\epsilon_1, \dots, \epsilon_{2\alpha})$ where
$\epsilon_i =0$ if $p_i$ came from $\mathbf{k} := \{k_1, \dots, k_{\alpha}\}$,
and $\epsilon_i=1$ if $p_i$ came from $\mathbf{l} := \{l_1, \dots, l_{\alpha}\}$.
Finally we define two new sequences $m_0,m_1, \dots, m_{2\alpha-1}$, and $\delta_1, \dots, \delta_{2\alpha -1}$,
where $m_0:=p_1$, $m_i= p_{i+1}-p_i$ and
$\delta_i= \epsilon_{i+1}-\epsilon_i$, for $i=1, \dots, 2\alpha -1$.
Notice that since we assume that $k_1 \leq l_1$, we have $p_1 = k_1$ and $\epsilon_1 = 0$. Let $v(\delta) := \sum_{i=1}^{2\alpha -1} |\delta|$, denote the \textsl{interlacement index}.
The terms with $v =1$ vanish, while the terms with $v=2$ will be considered separately.

We first consider the sum $I_n$ of the terms with $v \geq 3$ for which we drop the negative part and sum over the free index $m_0=k_1$ to obtain the bound
$$
I_n \leq c(\alpha)n\sum_{m_1, \dots,  m_{2\alpha-1}}
\sum_{x\in \mathbb{Z}^d} \prod_{t=1}^{2\alpha -1}
\,\sup_w\rP \big( S_{w,m_t}= \delta_t x\big),$$
where $c(\alpha)$ denotes generic constants depending only on $\alpha$, which may change from line to line.
Of these $2\alpha -1$ $\delta$'s, exactly $u:= 2\alpha -1 - v$ are equal to 0, and therefore
$$I_n \leq c(\alpha) n \Big[\sum_{i=0}^n \phi(i)\Big]^u
\sum_{j_1, \dots, j_{v}=0}^n \sum_{x} \prod_{t=1}^{v}
            \sup_{w_1,\ldots,w_v} \rP(S_{w_t,j_t} = \delta_t x).$$
Notice that if $S^{(1)}, \dots, S^{(v)}$ denote independent random walks then, assuming without loss of generality that $j_1 \leq \dots \leq j_v$,  we have that
\begin{align}
\sum_{x} \prod_{t=1}^{v} \rP(S_{w_t, j_t} = \delta_t x)
&\leq  \Big( \prod_{t=2}^{v-1} \max_{x} \rP( S^{(t)}_{j_t} = x) \Big)
\rP(S_{j_1}^{(1)} = \delta_v S_{j_v}^{(v)})\label{In}\\
&\leq \phi(j_1 + j_v)\prod_{t=2}^{v-1} \phi(j_t)
  \leq \prod_{t=2}^v \phi( j_t \vee j_1)\nonumber.
\end{align}
Writing $G_n:=\sum_{i=0}^n \phi(i)$, since $\phi$ is non-increasing 
we have that
\begin{align*}
\Delta_{n,v} &:= \!\!\sum_{0\leq j_1 \leq \cdots \leq j_v \leq n} 
\prod_{t=2}^v \phi(j_t \vee j_1)\leq \sum_{j_v=0}^n \phi(j_v) \!\!\sum_{0\leq j_1 \leq \cdots \leq j_{v-1} \leq n} 
\prod_{t=2}^{v-1} \phi(j_t \vee j_1) = G_n \Delta_{n,v-1},
\end{align*}
and repeating this procedure, for $v\geq 3$ we have that
$\Delta_{n,v}\leq \Delta_{n, 3} G_n^{v-3}$.
Combining the two bounds and summing over $v=3, \dots, 2\alpha -1$, we have the upper bound
$$\sum_{v=3}^{2\alpha -1}c(\alpha ) n G_n^{2\alpha -1 -v} \Delta_{n,v}
 \leq  c(\alpha ) n G_n^{2\alpha -1 -v + v -3} \Delta_{n,3} = c(\alpha ) n G_n^{2\alpha -4} \Delta_{n,3}.$$
Next we consider the sum $J_n$ over the terms with $v = 2$, which occurs when for some $j$, the indices $l_1 ,\dots, l_{\alpha}$ all lie in $[k_j, k_{j+1}]$.
Then it is easy to see that this sum $J_n$ is bounded above by
\begin{align*}
J_n 
&\leq c(\alpha)n \sup_{w_0,\ldots, w_{2\alpha-1}}
\sum_{m_{\alpha+1}, \cdots, m_{2\alpha-2}=0}^n 
\prod_{r=\alpha+1}^{2\alpha-2} \rP(S_{w_r,m_r} = 0) \\
&\quad \times 
\sum_{m_0, \cdots, m_\alpha =0}^n 
\bigg[\prod_{t=1}^{\alpha -1}
\rP(S_{w_t,m_t} = 0)\bigg]
 \Big[ \rP(S_{w_0,m_0}+S_{w_\alpha,m_\alpha}=0) - 
\rP(S_{w_0,m_0}+\ldots+S_{w_\alpha,m_\alpha}=0) \Big]\\
&\leq c(\alpha)n G_n^{\alpha -2}
\sup_{w_0,\ldots, w_{\alpha}}\\
&\quad \times\sum_{m_0, \cdots, m_\alpha=0}^n \bigg[\prod_{t=1}^{\alpha -1}
\rP(S_{w_t,m_t} = 0)\bigg] 
\Big[ \rP(S_{w_0,m_0}+S_{w_\alpha,m_\alpha}=0) - 
\rP(S_{w_0,m_0}+\ldots+S_{w_\alpha,m_\alpha}=0) \Big]\\
&\leq c(\alpha) n G_n^{\alpha -2} \sum_{m_0, \cdots, m_\alpha=0}^n
\bigg[ \prod_{t=1}^{\alpha -1} \phi(m_t) \bigg]\psi(m_0+m_{\alpha}, m_1 + \cdots + m_{\alpha-1})\\
&\leq c(\alpha) n G_n^{\alpha -2} A_{\psi}(1+A_{\psi})^{\alpha -2}
 \Bigg( \sum_{m_2, \cdots, m_{\alpha-1} } 
\prod_{t=2}^{\alpha -1}  \phi(m_t) \Bigg)
\times\sum_{m_0, m_1, m_{\alpha}} \phi(m_1) \psi(m_0+m_{\alpha}, m_1 ) \\
&\leq c(\alpha) A_{\psi} (1+A_{\psi})^{\alpha -2}  n G_n^{2\alpha -4} \sum_{i,j,k=0}^n \phi(j)\psi(i+k,j).\qedhere
\end{align*}
\end{proof}
The following corollary provides explicit bounds in the cases that are usually considered in the literature.
\begin{corollary}\label{cor_stan1}
Assume that the conditions of Proposition~\ref{pr:main} are satisfied with $\phi(m) = Tm^{-r}$ and
$\psi(m,k) = T m^{-r-1}(k\wedge m) $. Then, 
$$\var(L_n(\alpha)) \leq c(\alpha) T^{2\alpha -2}\times
\begin{cases}
n^2 \log(n) ^{2\alpha -4}, &\mbox{if $r=1$},\\
n^{4-2r}, &\mbox{if $1<r<3/2$,}\\
n \log(n), &\mbox{if $r=3/2$, and }\\
n, &\mbox{if $r>3/2$}.
\end{cases}
$$
\end{corollary}
Several relevant results treated so far in \cite{Beck09, Bolt89, CGP, Cerny07, Chen10,  DU11, Law91, Lewis93}
are not only obtained as a special case but also extended to the case of independent 
but not necessarily identically distributed variables, for example by applying the local limit theorem, as it is conducted in \cite{Law91}.

Also when $X_i$ is in the domain of attraction of the one-dimensional symmetric Cauchy law 
(\cite{DU10,DU11}), or in the case of strongly aperiodic planar random walk with second moments (\cite{Bolt89, CGP, Cerny07, Law91, Lewis93}), it is well known that the conditions of Proposition  
\ref{pr:main} are satisfied with $\phi(m) = T/m$ and $\psi(m,k) = T m^{-2}(k\wedge m) $.

However, we can do better for symmetrized variables and show that condition 
(\ref{eq:A}) implies (\ref{eq:B}), which together with the comparison technique motivate
the following results.


\begin{proposition}[Bound via comparison with symmetrised]\label{pr:add1}
Let $X_i$ be independent, $d$-dimensional random variables and $f_i(t):= \rE\exp(\ri t X_i)$, and assume that  there exists a non-negative measurable function $f(t),$ $0\leq f(t)\leq 1$ 
and positive non-increasing sequence $\phi(m)$ such that
\begin{eqnarray}\label{eq:C}
|1-f_i(t)|\leq	 Tf(t),\quad |f_i(\pm t)|\leq f(t), \quad\mbox{ and }\int_\Gamma f(t)^m dt\leq\phi(m),
\end{eqnarray}
for all $i, m \geq 0$, and $t\in \Gamma$. 
Then, for some constant $K=c(\alpha,d,T)$
\begin{equation*}
  \var(L_n(\alpha))
  \leq K n \Big( \sum_{i=0}^{n-1} \phi([i/2]) \Big)^{2\alpha -4} \sum_{j=0}^{n}j\phi([j/2])\sum_{k=j}^{{2}n} \phi([k/2])=:\Delta_n(\alpha, \phi).
\end{equation*}
\end{proposition}
\begin{proof}[Proof  of Proposition~\ref{pr:add1}]
Using the notation of Proposition~\ref{pr:main}, for positive integers $a,u, b,v$, with $a+u\leq b$, $\epsilon_j=\pm 1$ and any $x\in \mathbb{Z}^d$ 
\begin{align*}
\rP \Big( S_{a,u}+ (\epsilon,S_{b,v}) = x \Big) 
&\leq \frac{1}{(2\pi)^d} \int_\Gamma  \prod |f_j(\ep_j t)| \rd t 
\leq \frac{1}{(2\pi)^d} \int_\Gamma  f(t)^{u+v} \rd t
\leq \frac{1}{(2\pi)^d} \phi(u+v)
\end{align*}
To find $\psi(u,v)$, notice that since $f(t) \geq 0$, 
\begin{align*}
\phi(m)
&\geq \int_\Gamma f(t)^m \big[ 1-f(t)^m\big]\rd t 
=\sum_{j=0}^{m-1}\int_\Gamma f(t)^{m+j}(1-f(t))\rd t \geq m\int_\Gamma f(t)^{2m}(1-f(t))\rd t=:Q(2m)
\end{align*}
whence $Q(m)\leq\phi([m/2])/m$, where $[\cdot]$ denotes integer part. 
Therefore, 
\begin{align*}
\rP( S_{a,u} = 0) - \rP( S_{a,u}+ S_{b,1} = 0) 
&\leq CT\int_\Gamma f(t)^{u}(1-f(t)) dt \leq C T \phi([u/2])/u,
\end{align*}
and  it easily follows that (\ref{eq:B}) is satisfied with
$\psi(u,v) := \phi([u/2])\min(u,v)/u$. 
Thus all conditions of Proposition \ref{pr:main} are satisfied and the result follows from direct application  of (\ref{eq:C}).
\end{proof}
The following Corollary, allows for the case where $\phi(m)$ is regularly varying.
\begin{corollary}\label{cor_stan2}
Assume that the conditions of Proposition~\ref{pr:add1} 
are satisfied with $\phi(m) = h(m)m^{-r}$, $r\geq 1$,  
where $h(x)$ is a slowly varying at $x\to \infty$. Then, 
$$\var(L_n(\alpha)) \leq K\Delta_n(\alpha,\phi)\leq c_{\alpha} T^{2\alpha -2} 
\begin{cases}
n^2\Big[\sum_{k=1}^ n \frac{h(k)}{k}\Big]^{2\alpha -4}, 
	&\mbox{for $r=1$},\\
n^{4-2r}h^2(n),
	&\mbox{for $1<r<3/2$},\\
n\sum_{k=1}^ n h(k)^2/k,
	&\mbox{for $r=3/2$, and}\\
n,
	&\mbox{for $r>3/2$}.	
\end{cases}$$
\end{corollary}
Again,  the cited relevant results treated so far are not only obtained as a special case but also 
extended to dependent variables such as a random walk on a hidden Markov chain. In addition, following 
Kesten and Spitzer~\cite{Kest79} we can mimic the behaviour of genuinely $d$-dimensional random walk by constructing a one dimensional symmetric random walk with  characteristic function $f(t)=1-c|t|^{1/r}+o(|t|^{1/r})$ with $r=2/d$ for $d=2,3$ and $r=1/2$ for $d\geq 4$.

The following example of genuinely $2$-dimensional recurrent walk with infinite variance was motivated by Spitzer~\cite[pp. 87]{Spitzer76}.
\begin{example}\label{cor}
Let $S_n = \sum_{i=1}^n X_i$ be a random walk in $\mathbb{Z}^2$, such that
$\rP(|X|=k) = c/(k^3 \log(k)^\gamma)$, for $k\geq 4$ and $\gamma\in[0,1)$.
Then we have $\var(L_n(\alpha)) \leq cn^2 \max\{[\log n]^\gamma, \log \log n\}^{2\alpha-4}\log n^{-2(1-\gamma)}$, for $n\geq 10$.
Under these assumptions we have $\rP(S_n = 0)\leq c/n \log(n)^{1-\gamma}$, which is in the {\sl critical range}, where the random walk is recurrent, without second moment. 
To show it, we notice that by lengthy straightforward calculation the characteristic function of $X$  satisfies
(\ref{eq:C}) with 
\begin{align*}
\phi(n)&= \frac{c}{n \log(\re \vee n)^{1-\gamma}}\;,\; f(t)= \exp[- A |t|^2 h(|t|^2) ],\ \mbox{ where }\quad
h(r) := \big[ 1+ \log(1/r)_{+}\big]^{1-\gamma},
\end{align*}
and the sequence $\phi(m)$ is identified via  Fourier inversion, polar coordinates and a Laplace argument
\begin{equation*}
\int_{\Gamma} f(t)^n dt \leq c\int_0^1 \exp\big[-nr\big(1+\log(1/r) \big)^{1-\gamma} \big]  
+ O(\re^{-n}) \leq \frac{c}{n \log(\re \vee n)^{1-\gamma}} =:\phi(n).
\end{equation*}
\end{example}

\subsection{Bounds for identically distributed variables.}
\begin{proposition}[General upper bound for i.i.d.]\label{pr:add2} 
Let $X_i$ be i.i.d. $\mathbb{Z}^d$-valued random variables, 
and suppose that for all $n\in \mathbb{N}$, positive integers $a,u, b,v$, with $a+u\leq b$, and any $x\in \mathbb{Z}^d$
\begin{equation}\label{eq:iid}
\rP \Big( S_{a,u}\pm S_{b,v} = x \Big) \leq \phi(u+v),
\end{equation}
where $\phi(m)$ is a non-increasing sequence. Then,  for some constant $K=c(\alpha)$
\begin{equation*}
  \var(L_n(\alpha))
  \leq K n  \Big( \sum_{i=0}^{n-1} \phi(i) \Big)^{2\alpha -4} \sum_{j=0}^{n}j\phi(j)\sum_{k=j}^{[\alpha n]+1} \phi([k/\alpha])\;.
\end{equation*}
\end{proposition}
\begin{proof}[Proof  of Proposition~\ref{pr:add2}]
By inspecting the proof of  Proposition~\ref{pr:add1}, we notice that only need to bound the $J_n$ term.
Consider a typical ordering 
$$0\leq i_1 \leq \cdots \leq i_k \leq j_1 \leq \cdots \leq j_{\alpha} \leq i_{k+1} \leq \cdots \leq i_{\alpha} \leq n.$$
let us change variables to $(m_0,\dots, m_{2\alpha})$ such that $m_0+\cdots + m_{2\alpha}=n$. Then 
the contribution from this case to $J_n$ is
\begin{align}\label{c(n)term}
\sum_{m_0, \dots, m_{2\alpha}}
  \prod_{\substack{j\neq k, k+\alpha\\1\leq j \leq 2\alpha -1}}
  \rP(S_{m_j}=0) \Big[
    \rP(S_{m_k + m_{k+\alpha}} = 0) - \rP(S_{m_k + \cdots + m_{k+\alpha}}=0)
  \Big].
\end{align}
For $j\neq \alpha, k+\alpha$ keep $m_j$ fixed and sum over $m=m_k+m_{k+\alpha}$, from $0$ to 
$M$ which depends on $n$, and the $m_j$ for $j\neq k, k+\alpha$.
Then for given $m_{k+1}, \dots, m_{k+\alpha -1}$, the term in the sum is 
$$
\sum_{m=0}^M (m+1) [\rP(S_m=0)-\rP(S_{m+q}=0)],$$
where $q:=m_{k+1}+\ldots+m_{k+\alpha-1}$.
Then since $M\leq n-q$, it is an easy exercise to show that this sum is bounded above by
\begin{align*}
\lefteqn{\sum_{m=0}^{M} (m+1) \Big[\rP(S_m=0)-\rP(S_{m+q}=0)\Big]}\\
&\leq 
\sum_{m=0}^{q-1} (m+1)\rP(S_m=0) +q \mathbf{I}(n-q\geq q)\sum_{m=q}^{n-q} \rP(S_{m}=0)\\
&\leq \sum_{m=0}^{(\alpha m^{\ast})\wedge n} (m+1) \rP(S_m=0) +\alpha m^{\ast} \sum_{m=m^{\ast}}^{n} P(S_{m}=0)
\end{align*}
where $m^{\ast}:=\max\{m_{k+1},\ldots,m_{k+\alpha-1}\}$. 
%
The result follows by summing over all indices apart from $m^\ast$ and changing the order of summation.
\end{proof}

\subsection{Proofs of main results}

\begin{proof}[Proof of Theorem~\ref{thm:dbound}.]
We apply a comparison argument found to be useful in many areas (e.g. Pruss and Montgomery-Smith~\cite{Mont01}, and Lefevre and Utev~\cite{LU03}), more exactly, we bound $\var(L_n)$  by the corresponding characteristic for the symmetrised random walk.

Following Spitzer's argument 
we notice that with $f(t)=\rE[\exp(\ri t\cdot X_1)]$
\begin{align*}
\rP \Big( S_{a,u}+\ep S_{b,v} = x \Big) &\leq c \int_{\Gamma} |f(t)|^u |f(-t)|^v \rd t
= c \int_{\Gamma} \Big[|f(t)|^2\Big]^{u/2}\Big[|f(-t)|^2\Big]^{v/2}\rd t
\end{align*}
Since $|f(t)|^2$ is a characteristic function of $d$-dimensional symmetric integer variable,  for some positive $\lambda$, 
$
1-|f(t)|^2\geq \lambda |t|^2$,
and hence, 
$$
\rP \Big( S_{a,u}+\ep S_{b,v} = x \Big) \leq c \int_{\Gamma}  
\exp\Big[-\frac{\lambda (u+v)}{2}|t|^2  \Big]\rd t
\leq c (u+v)^{-d/2}$$
and the proof follows from Proposition \ref{pr:add2} applied with $\phi(m)= m^{-d/2}$.
\end{proof}
The proof of Theorem~\ref{thm:nasc} will be based on the following Lemma.
\begin{lemma}\label{lm:extra}
Assume $X$ is genuinely $d$-dimensional and $\rE |X|^2=\infty$. Then there exists a monotone slowly varying function $h_n\to 0$ as $n\to \infty$ such that 
\bb
\sup_{x\in\mathbb{Z}^d} \rP(S_n=x) \leq c_d \int_\Gamma |\rE e^{it\cdot X}|^n d t\leq h_n n^{-d/2} \ee
\end{lemma}
\begin{proof}[Proof of lemma~\ref{lm:extra}.]
Without loss of generality assume that $X$ is symmetrized. Let 
$\sigma_{e,L}:=\rE \big[(e\cdot X)^2 \mathbf{I}(|X|\leq L)\big]$.
Following Spitzer, since $X$ is genuinely $d$-dimensional, we may assume that there exist positive constants $c$ and $W$ such that for any unit  vector $|e|=1$,\
$\sigma_{e,W}\geq c$ and $1-f(t)\geq c |t|^2$.
Let $\lambda_d$ be the $d$-dimensional Lebesgue measure on $\mathbb{R}^d$, and $\mu_d$ the Lebesgue-Haar  measure on $S^{d-1}:=\{e\in [-\pi,\pi]^d: |e|=1\}$.   
Notice that since  $\rE |X|^2=\infty$, for any $K$ we have
$\mu_d\{e: \sigma_{e,\infty} <K\}=0$. 

Fix a small positive $x$ such that $\sqrt{c/x}\geq 2W$, and
for any $\epsilon>0$ let $K=K(\epsilon)=\epsilon^{-d/2}$. Then there exists $L=L(\epsilon)>0$ small enough so that
$\mu_d\{e: \sigma_{e,L} <K\}\leq \ep^{d/2}$.
We partition $S^{d-1}$ in two sets
\bb
A_{L,K}=\{e\in S^{d-1}:\sigma_{e,L}\geq K\}
\quad\mbox{and}\quad \bar{A}_{L,K}=\{e\in S^{d-1}:\sigma_{e,L}<K\}\;,\ee
so that for any direction $e\in \bar{A}_{L,K}$,
\bb
\{z\in \mathbb{R}: 1-f(ze)\leq x\}\subseteq \{z:c z^2\leq x\}\subseteq \{z: |z|\leq \sqrt{x/c}\}\;.\ee
Hence, using  $d$-dimensional spherical coordinates, 
\begin{align*}
\lambda_d\{(z,e) \in \mathbb{R}\times\bar{A}_{L,K}: 1-f(ez)\leq x\}
&\leq \mu_d\{\bar{A}_{L,K}\}(x/c)^{d/2} (1/d)
\leq \ep^{d/2}(x/c)^{d/2} (1/d)\;.
\end{align*}
On the other hand, for any $t$, 
\begin{align*}
1-f(t)
&=2\sum_{k\in Z^d} \sin([t\cdot k]/2)^2 P(X=k)\geq (1/4) E[(t\cdot X)^2I(|t\cdot X|\leq 1/2)]=(|t|^2/4) \sigma_{t/|t|, 1/2|t|}.
\end{align*}
Now, assume that $\sqrt{c/x}\geq 2L\;$.
Then for any direction $e \in A_{L,K}$, by choice of $x$ and since $\sigma_{e,L}$ is increasing in $L$, for $cz^2\leq 1-f(ez)\leq x$ or $|z|\leq \sqrt{x/c}$, it must be the case that
\begin{equation*}
x\geq 1-f(ez) \geq (z^2/4) \sigma_{e, 1/2z}
\geq (z^2/4) \sigma_{e, L}
\geq (z^2/4) K
\end{equation*}
implying that on the set  $A_{L,K}$, it must be that $|z|\leq 2\sqrt{x/K}$.
Changing to  $d$-dimensional polar coordinates, we find that
\begin{align*}
\lambda_d\Big\{(z,e)\in \mathbb{R}\times A_{L,K} : 1-f(ez)\leq x\Big\}
&\leq \int_{A_{L,K}} \int_0^{ \sqrt{4x/K}} r^{d-1}\rd r \rd e
\leq C_d \epsilon^{d/2} x^{d/2}
\;.
\end{align*}
Overall, for $x\leq c/4L^2$, $\lambda_d\{t: 1-f(t)\leq x\}\leq c_d (x\ep)^{d/2}$, and hence
$\{t\in \Gamma: 1-f(t)\leq x\}$
has Lebesgue measure $o(x^{d/2})$.

Let $F(x)$ be the cumulative distribution function of $\log (1/f(t))$ on the probability space 
$\Gamma$ with normalised  Lebesgue measure. Then $F$ is continuous at $x=0$ and supported on $\mathbb{R}^+$. Moreover, as $ 0<x\to 0$,
$F(x) =o(x^{d/2})$. 
Therefore, for some positive sequence $\epsilon_n\to 0$

$$
\frac{1}{[2\pi]^2} \int_\Gamma f(t)^n \rd t = \int_0^\infty e^{-nx} \rd F(x)
=n\int_0^\infty \re^{-nx} F(x)\rd x 
\leq n^{-d/2} \epsilon_n. $$

It remains to show that there exists a  positive monotone slowly varying function $\ep_n\leq h(n)\to 0$  as $n\to \infty$.
Let $\delta_n=\sup_{j\geq n}\ep_j$, $a_0:=0$ and for $n\geq 1$ define $a_n$ recursively by
$a_n=\min(2a_{2^{r-1}}, 1/\delta_n)$, for $2^{r-1}< n\leq 2^r$, so that $a_n\to \infty $ is monotone, $a_{2^{r}}\leq 2a_{2^{r-1}}$ implying that
 $a_{2n}\leq 4a_{n}$, and $1/a_n\geq \delta_n\geq \ep_n$. Finally, take $h_n:= 1/\max(a_0, \log a_n)$.
\end{proof}
\begin{proof}[Proof of Theorem~\ref{thm:nasc}.]
Assume that $\rE|X|^2=\infty$ and $d=2$ or $d=3$. Then, by Lemma~\ref{lm:extra} there exists a slowly varying function $h(n)\to 0$ as $n\to \infty$ such that 
$\int_\Gamma |\rE\exp(\ri t\cdot X )|^n \rd t\leq h_n n^{-d/2}$.
Applying Corollary~\ref{cor_stan2} with $r=1$ and $r=3/2$ we respectively find that
$$
\var(L_n(\alpha)) \leq 
\begin{cases}
K n^2\Big(\sum_{k=1}^ n h(k)/k\Big)^{2\alpha -4} =o(n^2 (\log n)^{2\alpha-4}),&\mbox{for $d=2$, and}\\
K n\Big(\sum_{k=1}^ n h(k)^2/k\Big) = o(n\ln n),&\mbox{for $d=3$}.
\end{cases}$$

Finally assume that $\rE |X|^2 < \infty$ and $E[X]=\mu\neq 0$. Then
$\rP(S_n=0) = \rP(S'_n=-n\mu)$ whence it follows that $\rP(S_n=0)=o(n^{-d/2})$ (see for example \cite[Theorem~2.3.10]{Law10}). Then inspecting the proof of Proposition 4, one can readily obtain the desired bound for the $J_n$ term, while with slight modification the bound for the $I_n$ term also follows.

Note that for $d=1$ the situation is much simpler since then $\var(L_n(\alpha|\epsilon, d)) \sim C [\rE L_n(\alpha|\epsilon, d)]^2$ and if 
$\rE|X|^2=\infty$ or $\rE[X] \neq 0$, $\rE L_n(\alpha|\epsilon, d) =o(n^{(1+\alpha)/2})$.
\end{proof}
\begin{proof}[Proof of Theorem~\ref{thm:asymptotic}.]
We first give the proof for the case $d=1$. 
As in the proof of Proposition~\ref{pr:main} we begin from expression \eqref{eq:var}, and define the sequences $p_i$, and $\delta_i$ for $i=1, \dots, 2\alpha -1$, and the quantity $v(\delta) = \sum_{i=1}^{2\alpha -1} |\delta_i|$. 
Recall that $v(\delta)$ measures the interlacement of 
the two sequences $k_1, \dots, k_{\alpha}$, and $l_1, \dots, l_{\alpha}$. 
For example $v(\delta) = 1$ occurs when either $k_{\alpha} \leq l_1$, or $l_{\alpha} \leq k_1$, in which case the contribution vanishes by the Markov property.
On the other hand $v(\delta)=2$ when for example $l_1, \dots, l_{\alpha} \in [k_i, k_{i+1}]$ for some $i$. Finally $v(\delta) = 3$ occurs when for example
$$k_1 \leq \cdots \leq k_{r} \leq l_1 \leq \cdots \leq l_{s} \leq k_{r+1} \leq \cdots \leq k_{\alpha} \leq l_{s+1} \leq \cdots \leq l_{\alpha} \leq n.$$
From the proof of Proposition~\ref{pr:main}, and using the bound $\rP(S_n = 0) =O(1/n)$, the terms of the sum are bounded above by
$ n^2 \log(n)^{2\alpha -1 -v(\delta)}$, and thus the leading term appears when either $v(\delta) = 2, 3$, with other terms giving strictly lower order.
We shall therefore analyze these two situations in detail in order to derive the exact asymptotic constants. 
When $v = 3$, the two terms in the difference individually give the correct order and shall be treated by the classical Tauberian theory.
However for $v=2$, the two terms only give the correct order when considered together. This however forbids the use of Karamata's Tauberian theorem since the monotonicity restriction would require roughly that $X_i$ is symmetrized. Thus the complex Tauberian approach, as developed in \cite{DU11}, is required to 
justify the answer. 
\vskip3pt
\noindent{\bf \large Case 1:} $v(\delta) = 3$.
Assume that part of the sequence $\mathbf{l} = \{l_1, \dots, l_{\alpha}\}$ lies between $k_r$ and $k_{r+1}$, and the rest between $k_s$ and $k_{s+1}$. Then using the change of variables
\begin{gather*}
    i_1 = m_0, i_2= m_0+m_1, \cdots, i_r = m_0+\cdots+m_{r-1}\\
    j_1 = m_0 + \cdots+m_r, j_2 = m_0+\cdots + m_{r+1}, \cdots, j_s= m_0+ \cdots + m_{r+s-1},\\
    i_{r+1} = m_0+\cdots + m_{r+s}, i_{r+2} = m_0+ \cdots + m_{r+s+1}, \cdots, i_{\alpha} = m_0+ \cdots + m_{\alpha + s -1}\\
    j_{s+1} = m_0+ \cdots + m_{\alpha + s}, j_{s+2}=m_0+\cdots+m_{\alpha + s +1}, \cdots, j_{\alpha} = m_{2\alpha -1}, 
    n=m_0+ \cdots + m_{2\alpha}.
\end{gather*}
we rewrite the positive term in \eqref{eq:var} as
\begin{align*}
a(n)
&= \sum\rP\bigg[ S(i_1) = \cdots=S(i_{\alpha}); S(j_1) = \cdots = S(j_{\alpha})\bigg]\\
&=\sum_{m_0, \cdots, m_{2\alpha -1}} \Big[\prod_{\substack{j=1 \\ j\neq r, r+s, \alpha +s}}^{2\alpha -1}
\rP(S_{m_j}=0)\Big]
\times \rP(S_{m_r} + S'_{m_{r+s}} = S'_{m_{r+s}}+ S''_{m_{\alpha +s}}=0).
\end{align*}
Notice that from \cite{DU11} we have that $\sum_{n\geq 0} \lambda^n \rP(S_n = 0) 
\sim  \log \big(1/(1-\lambda) \big)/\pi \gamma$. Let 
$$a(\lambda) = (1-\lambda)^{-3} [-\log(1-\lambda)]^{2\alpha -4}\;, \; c_\gamma = 
(\pi \gamma)^{-2\alpha +4} .$$
Then, by direct calculations and Fourier inversion formula 
\begin{align*}
 \sum_{n\geq 0 }\lambda^n a(n) &= c_\gamma (1-\lambda) a(\lambda) 
\sum_{x\in \mathbb{Z}} \sum_{k_1, k_2, k_3\geq 0} \lambda ^{k_1+ k_2+ k_3} 
\rP(S_{k_1} = x) \rP(S_{k_2} =- x) \rP(S_{k_3} = x)\\
&=  c_\gamma (1-\lambda) a(\lambda)  \frac{1}{(2\pi)^2} 
    \iint_{[-\pi, \pi]^2}
        \frac{\rd t \rd s}{(1-\lambda f(t))(1-\lambda f(s)) (1-\lambda f(t+s))}\\
& \sim c_\gamma (1-\lambda) a(\lambda)
\frac{1}{(2\pi)^2\gamma^2} \frac{1}{1-\lambda}
    \iint_{\mathbb{R}^2}
    \frac{\rd x \rd y}{(1+|x|)(1+|y|)(1+|x+y|)}\sim (1/4\gamma^2) c_\gamma  a(\lambda)
\end{align*}
Next we consider the negative term in \eqref{eq:var}
\begin{multline*}
b(n)
:= \sum_{m_0, \dots, m_{2\alpha -1}}
\rP\Big[ S_{m_1} = \cdots = S_{m_{r-1}} = S_{m_r} + \cdots + S_{m_{r+s}} = S_{m_{r+s+1}}=\cdots = S_{m_{\alpha + s -1}}=0 \Big]\\
 \times
    \rP\Big[ S_{m_{r+1}} = \cdots = S_{m_{r+s}} + \cdots + S_{m_{\alpha +s}} = S_{m_{\alpha + s + 1}} = \cdots = S_{m_{2\alpha -1}} = 0\Big].
\end{multline*}
By direct calculations and \eqref{eq:expansion1}, 
\begin{multline*}
\sum_n \lambda^n b(n) =
 \Big(\frac{1}{\pi \gamma} \log\big( \frac{1}{1-\lambda}\big)\Big)^
 {\alpha -s+r -2}(1-\lambda)^{-2} \sum_{m_{r}, \dots, m_{\alpha +s}=0}^\infty \lambda^{m_r + \cdots + m_{\alpha +s}} \\
\times 
\prod_{\substack{t=r+1, \cdots, \alpha +s -1 \\ t\neq r+s }} \rP (S_{m_t}=0)
\rP(S_{m_{r}}+ \cdots + S_{m_{r+s}}=0) \rP(S_{m_{r+s}} + \cdots S_{m_{\alpha +s}}=0),
\end{multline*}
and using Fourier inversion and \eqref{eq:expansion1} the internal sum behaves as
\begin{align*}
&\lefteqn{(2\pi)^{-\alpha-s+r} \int_{-\pi}^{\pi}\!\!\!\!\!\cdots\!\!\int_{-\pi}^{\pi}
(1-\lambda \phi(x))^{-1}  (1-\lambda \phi(x)\phi(y))^{-1}  (1-\lambda \phi(y))^{-1}}
\\
&\qquad\qquad  \times \Big[ \prod_{j=r+1}^{r+s-1} \prod_{k=r+s+1}^{\alpha+s-1}
  (1-\lambda \phi(x)\phi(t_j))^{-1} (1-\lambda \phi(y) \phi(t_k))^{-1}
\rd t_{j} \, \rd t_k \Big] \rd x \, \rd y\\
&\sim 
(\pi\gamma)^{-\alpha-s+r}(1-\lambda)^{-1}
\log\big(\frac{1}{1-\lambda}\big)^{\alpha -r +s -2} \frac{\pi^2}{6}.
\end{align*}
Then we have $\sum_n \lambda^n b(n)\sim 
(\pi^2/6 (\pi \gamma)^{2\alpha -2}) a(\lambda)$,
whence the Tauberian theorem implies that 
$a(n)-b(n) \sim n^2 \log(n)^{2\alpha -4} /24 \pi^{2\alpha -4} \gamma^{2\alpha -2}$.
\noindent
Most importantly we see that the lengths and locations of the chains, $r$ and $s$, do not affect the asymptotic. 
Noting that if $1\leq r, s \leq \alpha -1$, we can partition $2\alpha= r + s + (\alpha -r) + (\alpha  -s)$ in $(\alpha-1)^2$ ways, and thus overall the total contribution from terms with $v=3$ is
$$[(\alpha !(\alpha-1))^2/12 \pi^{2\alpha -4} \gamma^{2\alpha -2}] n^2 \log(n)^{2\alpha -4}.$$
\vskip 3pt
\noindent{\bf Case 2:} $v(\delta) = 2$. The typical term $c(n)$ was introduced in (\ref{c(n)term}) in 
the proof of Proposition \ref{pr:add2}. 
Now we let $\lambda\in \mathbb{C}$, with $|\lambda|<1$. 
By lengthy but direct calculations we can derive an expression of the form
\begin{align*}
\sum_n \lambda^n c(n) = \frac{\alpha-1}{(\gamma \pi)^{2\alpha-2}} a(\lambda)+ o(a(\lambda)), \quad \lambda \to 1.
\end{align*}
The approach developed in \cite{DU11} can then be used to bound the error terms and show that 
\ $c(n) \sim [(\alpha-1)/2(\gamma \pi)^{2\alpha-2}] n^2 \log(n)^{2\alpha -4}.$ 

Finally taking into account the fact that the $l_1, \dots, l_\alpha$ can be in any of the $\alpha -1$ intervals $[k_i, k_{i+1}]$, for $i=1, \dots, \alpha -1$, 
the result follows the overall contribution of terms with $v(\delta) = 2$ is
$$\frac{(\alpha-1)^2}{2(\gamma \pi)^{2\alpha-2}} n^2 \log(n)^{2\alpha -4}.$$

The case for $d=2$ is very similar, so we move on to the case $d=3$. 

\noindent{\bf Case $d=3$, $\alpha =2$.}
Using the same notation as before, we have three terms to consider $a(n)$, $b(n)$, and $c(n)$. We first consider $c(n)$. Letting $K := \epsilon /\sqrt{1-\lambda}$ and using the usual power series construction and
spherical coordinates
\begin{align}\label{eq:kappa1}
\sum_{n} \lambda^n c(n)\notag
&= (1-\lambda)^{-2} (2\pi)^{-6} \iint_{J^3\times J^3}
\frac{\lambda f(y) (1-f(x)) \rd x \rd y}
{(1-\lambda f(x))^2 (1-\lambda f(y)) (1-\lambda f(x)f(y))}\notag\\
&\sim 2(2\pi)^{-4} |\Sigma|^{-1}(1-\lambda)^{-2} 
\int_0^K \int_0^K \frac{r^4 s^2 \rd r \rd s}
{(1+r^2)^2(1+s)^2(1+r^2+s^2)}\notag\\
&\sim 2(2\pi)^{-4} |\Sigma|^{-1}\frac{\pi}{2} (1-\lambda)^{-2} \log\Big( \frac{1}{1-\lambda}\Big)
=: \kappa_1 (1-\lambda)^{-2} \log\Big( \frac{1}{1-\lambda}\Big),
\end{align}
and thus $c(n) \sim \kappa_1 n \log n$, 
where $\kappa_1 >0$, where the answer can be justified following \cite{DU11}.

The term $a(n)-b(n)$ is trickier to compute. As usual we consider the power series
\begin{align*}
\sum_{n\geq 0} \lambda^n \big( a(n) - b(n)\big)
&= (1-\lambda)^{-2} (2\pi)^{-6} 
\iint_{B(\epsilon)} 
\frac{ \rd x \rd y }{(1- \lambda f(x))(1-\lambda f(y))(1-\lambda f(x+y))}
\\
&\qquad - (1-\lambda)^{-2} (2\pi)^{-6} 
\iint_{B(\epsilon)} 
\frac{ \rd x \rd y }{(1- \lambda f(x))(1-\lambda f(y))(1-\lambda f(x)f(y))}\\
&= (1-\lambda)^{-2} (2\pi)^{-6} (I_1(\lambda) - I_2(\lambda)).
\end{align*}
Let $A\in [-1,1]$ be the cosine of the angle between $x$ and $y$, which  in spherical coordinates is
\begin{equation}A = A(\theta_1,\theta_2,\phi_1,\phi_2) = \cos(\phi_1-\phi_2)\sin(\theta_1) \sin(\theta_2) + \cos(\theta_1)\cos(\theta_2).\label{eq:defA}
\end{equation}
Then as $0<\lambda \uparrow 1$, using the expansion \eqref{eq:expansion1}
\begin{align*}
I_1(\lambda)
&\sim|\Sigma|^{-1} \int_{r,s=0}^{\epsilon}
\int_{\phi_{1,2}=0}^{2\pi}
\int_{\theta_{1,2}=0}^{\pi}
\frac{r^2 s^2 \sin(\theta_1)\sin(\theta_2) \rd r \rd s \rd \theta_1 \rd \theta_2 \rd \phi_1 \rd \phi_2}
{(1-\lambda + \lambda r^2)(1-\lambda + \lambda s^2)
\Big[1-\lambda + \lambda \big( r^2 + s^2 + 2 Ars\big)\Big]}\\
&= |\Sigma|^{-1} 
\int_{\mathbf{\phi}, \mathbf{\theta}}\sin(\theta_1)\sin(\theta_2) 
\int_{r=0}^K \int_{s=0}^K
\frac{r^2 s^2 \rd r \rd s }
{(1+ r^2)(1+ s^2)
\Big[1+ r^2 + s^2 + 2 Ars\Big]}
\rd \mathbf{\theta} \rd \mathbf{\phi}\\
&\sim 
|\Sigma|^{-1} \log(K)
\int_{\mathbf{\phi}, \mathbf{\theta}}
\sin(\theta_1)\sin(\theta_2) 
\frac{ \arccos(A(\mathbf{\theta}, \mathbf{\phi})) }
{\sqrt{1-A(\mathbf{\theta}, \mathbf{\phi})^2}} .
\end{align*}
The other integral is slightly easier
\begin{align*}
I_2(\lambda)
&\sim |\Sigma|^{-1}\frac{\pi}{2} \log{K} \int_{\theta, \phi} \sin(\theta_1)\sin(\theta_2) \rd\theta_1\rd\theta_2 \rd \phi_1 \rd \phi_2,
\end{align*}
and thus overall we must have that 
\begin{align}
(I_1- I_2)(\lambda)
&\sim 
\frac{1}{2}(2\pi)^{-6}|\Sigma|^{-1} 
(1-\lambda)^{-2} \log\Big( \frac{1}{1-\lambda} \Big)\notag\\
&\quad \times\int_{\theta_1, \theta_2=0}^{\pi} \int_{\phi_1,\phi_2=0}^{2\pi} 
\Big[ \frac{\arccos(A)}{\sqrt{1-A^2}} - \frac{\pi}{2} \Big] \sin(\theta_1)\sin(\theta_2) \rd \theta_{1,2} \rd \phi_{1,2} \notag\\
&=: \kappa_2 (1-\lambda)^{-2} \log\Big( \frac{1}{1-\lambda} \Big),\label{eq:kappa2}
\end{align}
whence it follows that $\var(L_n(2)) \sim (\kappa_1 + \kappa_2) n \log n$.

To prove the last claim, let $S'_n = X'_1 + \cdots + X'_n$ be another random walk, independent of $S_n$, such that its characteristic function $f'(t)= \rE[\exp(\ri t X'_i)]$ also satisfies the expansion \eqref{eq:expansion1}. Then using \cite[Lemma~3.1]{DU11} one can adapt the proof of \cite[Theorem~2.1]{DU11} to show that 
$L'_n(\alpha) = L_n(\alpha) + o (L_n(\alpha))$.
\end{proof}


\begin{thebibliography}{00}
\bibitem{Aaronson} J.~Aaronson.
\newblock Relative complexity of random walks in random sceneries. 
\newblock \emph{Ann. Probab.} 40, 2012, no. 6, pp. 2460-2482.
%
%

\bibitem{Asselah}
A.~Asselah. 
\newblock Shape transition under excess self-intersections for transient random walk.
\newblock \emph{Annales de l'institut Henri Poincare (B) Probability and
  Statistics}, 
46, no. 1, pp. 1250-278, 2010.

\bibitem{Beck09}
M.~Becker and W.~K{\"o}nig.
 Moments and distribution of the local times of a transient random
  walk on {$\mathbb{Z}^d$}.
 \emph{J. Theoret. Probab.}, 22 (2): 365--374, 2009.

\bibitem{Bolt89}
E.~Bolthausen.
 A central limit theorem for two-dimensional random walks in random
  sceneries.
 \emph{Ann. Probab.}, 17 (1): 108--115, 1989.

\bibitem{Bor79}
A.~N. Borodin.
 A limit theorem for sums of independent random variables defined on a
  recurrent random walk.
 \emph{Dokl. Akad. Nauk SSSR}, 246 (4): 786--787,
  1979.


\bibitem{Bry95}
D.~C. Brydges and G.~Slade.
 The diffusive phase of a model of self-interacting walks.
 \emph{Probab. Theory Related Fields}, 103 (3):
  285--315, 1995.

\bibitem{CGP}
F.~Castell, N.~Guillotin-Plantard, and F.~P\`ene.
 {Limit theorems for one and two-dimensional random walks in random
  scenery.}
 \emph{Annales de l'institut Henri Poincare (B) Probability and
  Statistics}, 2012.

\bibitem{Chen10}
X.~Chen.
 \emph{Random walk intersections}, volume 157 of \emph{Mathematical
  Surveys and Monographs}.
 American Mathematical Society, Providence, RI, 2010.
 Large deviations and related topics.

\bibitem{DU10}
G.~Deligiannidis and S.~Utev.
 An asymptotic variance of the self-intersections of random walks,
  2010,  arXiv:1004.4845.

\bibitem{DU11}
G.~Deligiannidis and S.~Utev.
 Computation of the asymptotics of the variance of the number of
  self-intersections of stable random walks using Wiener-Darboux theory.
 \emph{Sib. Math. J}, 52, 2011.

%
\bibitem{DZ15}
G.~Deligiannidis and K.~Zemer. 
\newblock Relative complexity of random walks in random scenery in the absence of a weak invariance principle for the local times.
\newblock preprint, 2015.


\bibitem{GS09}
J{\"u}rgen G{\"a}rtner and Rongfeng Sun.
 A quenched limit theorem for the local time of random walks on {$\Bbb
  Z^2$}.
 \emph{Stochastic Process. Appl.}, 119 (4):
  1198--1215, 2009.


\bibitem{Kest79}
H.~Kesten and F.~Spitzer.
 A limit theorem related to a new class of self-similar processes.
 \emph{Z. Wahrsch. verw. Gebiete}, 50: 5--25, 1979.

\bibitem{Law91}
G.F. Lawler.
 \emph{{Intersections of Random Walks}}.
 Birkhauser, 1991.

\bibitem{Law10}
Gregory~F. Lawler and Vlada Limic.
 \emph{Random walk: a modern introduction}, volume 123 of
  \emph{Cambridge Studies in Advanced Mathematics}.
 Cambridge University Press, 2010.

\bibitem{LU03}
C.~Lef{\`e}vre and S.~Utev.
 Exact norms of a {S}tein-type operator and associated stochastic
  orderings.
 \emph{Probab. Theory Related Fields}, 127 (3):
  353--366, 2003.

\bibitem{Lewis93}
T.~M. Lewis.
 A law of the iterated logarithm for random walk in random scenery
  with deterministic normalizers.
 \emph{J. Theoret. Probab.}, 6 (2): 209--230, 1993.

\bibitem{Mont01}
S.~J. Montgomery-Smith and A.~R. Pruss.
 A comparison inequality for sums of independent random variables.
 \emph{J. Math. Anal. Appl.}, 254 (1): 35--42, 2001.

\bibitem{Spitzer76}
F.~Spitzer.
 \emph{{Principles of Random Walk }}.
 Springer, 1976.

\bibitem{Cerny07}
J.~\v{C}ern{\`y}.
 {Moments and distribution of the local time of a two-dimensional
  random walk}.
 \emph{Stochastic Process. Appl.}, 117 (2):
  262--270, 2007.
\end{thebibliography}
\end{document}